\documentclass{amsart}
\usepackage{amsmath,amssymb,amstext,amsopn,amsthm,lmodern,hyperref}
\usepackage[T1]{fontenc}
\usepackage[utf8]{inputenc}
\usepackage[polish,english]{babel}
\usepackage[capitalize,nameinlink]{cleveref}
\crefname{equation}{}{}
\Crefname{figure}{Figure}{Figures}
\crefname{page}{page}{pages}
\Crefname{enumi}{}{}
\Crefname{subsection}{Subsection}{Subsections}
\def\theoremname{Theorem}%
\def\propositionname{Proposition}%
\def\lemmaname{Lemma}%
\def\corollaryname{Corollary}%
\def\definitionname{Definition}%
\def\conventionname{Convention}%
\def\axiomname{Axiom}%
\def\remarkname{Remark}%
\def\examplename{Example}%
\def\constructionname{Construction}%
\def\notationname{Notation}%
\def\assumptionname{Assumption}%
\def\conjecturename{Conjecture}%
\newtheorem{thm}{\theoremname}[section]
\theoremstyle{plain}
\newtheorem{theorem}[thm]{\theoremname}
\newtheorem{proposition}[thm]{\propositionname}
\newtheorem{lemma}[thm]{\lemmaname}

\theoremstyle{definition}

\newtheorem{example}[thm]{\examplename}
\newtheorem{question}[thm]{Question}

\newcommand{\field}[1]{\mathbb{#1}}

\newcommand{\abs}[1]{\lvert#1\rvert}

\newcommand{\leb}{\mathcal{L}^1}

\newcommand{\ignoreSpellCheck}[1]{#1}

\makeatletter
\DeclareRobustCommand{\crefnosort}[1]{%
  \begingroup\@cref@sortfalse\cref{#1}\endgroup
}
\makeatother

\usepackage{pgfplots}
\usetikzlibrary{calc}
\usetikzlibrary{math}
\pgfplotsset{compat=1.14}

\begin{document}
	
\title[Attractor of Cantor type with positive measure]{Attractor of Cantor type with positive measure}

\author[J.~Morawiec]{Janusz Morawiec}
\address{Instytut Matematyki\\
Uniwersytet \'{S}l\c aski\\
Bankowa 14\\
PL-40-007 Katowice\\
Poland}
\email{morawiec@math.us.edu.pl}
\author[T.~Z\"{u}rcher]{Thomas Z\"{u}rcher}
\address{Instytut Matematyki\\
Uniwersytet \'{S}l\c aski\\
Bankowa 14\\
PL-40-007 Katowice\\
Poland}
\email{thomas.zurcher@us.edu.pl}


\begin{abstract}
We construct an iterated function system consisting of strictly increasing contractions $f,g\colon [0,1]\to [0,1]$ with $f([0,1])\cap g([0,1])=\emptyset$ and such that its attractor has positive Lebesgue measure.
\end{abstract}

\maketitle

\renewcommand{\theequation}{\arabic{section}.\arabic{equation}}\setcounter{equation}{0}


\section{Introduction}
Self-similar measures associated with iterated function systems (shortly IFS) have many significant and interesting applications in various areas of science, including mathematics, and in particular, the theory of functional equations (see e.g.~\cite{PS1998,M2009,CC2016}). Studying a functional equation connected with the problem posed in \cite{M1985,MZ}, we come to the following question.

\begin{question}\label{IFS: question}
Consider an iterated function system consisting of strictly monotone contractions $f_1,\ldots,f_N\colon [0,1]\to[0,1]$ such that
\begin{equation}\label{condition}
\bigcup_{n=1}^Nf_n([0,1])\neq[0,1]\quad\hbox{and}\quad
f_i((0,1))\cap f_j((0,1))=\emptyset\quad\hbox{for all }i\neq j.
\end{equation}
Is the attractor of this IFS necessary of Lebesgue measure zero? 	
\end{question}

Surprisingly we could not find any answer to this question by looking through the literature in this topic. The purpose of this paper is to give a negative answer to this question by constructing an example of an IFS consisting of two strictly increasing contractions $f,g\colon [0,1]\to [0,1]$ such that $f([0,1])\cap g([0,1])=\emptyset$ with the attractor of positive Lebesgue measure.


\section{Preliminaries}
We say that a function $f\colon [a,b]\to\mathbb R$ is {\it $L$-Lipschitz}, if
\begin{equation*}
\abs{f(x)-f(y)}\leq L\abs{x-y}\quad\hbox{for all } x,y\in [a,b].
\end{equation*}
Any $L$-Lipschitz function $f\colon [a,b]\to\mathbb R$  with $L<1$ is said to be a {\it contraction}.

The following fact will be used frequently in the announced construction; its proof is simple, so we omit it.

\begin{lemma}\label{Lipschitz: from local to global}
Let $f\colon [a,b]\to \field{R}$ be a function and let $c\in [a,b]$.
If the restrictions of~$f$ to $[a,c]$ and to~$[c,b]$ are both $L$-Lipschitz, then $f$ is $L$-Lipschitz as well.
\end{lemma}

In this paper, whenever we are given a finite collection of contractions defined on the interval $[0,1]$ into itself, we refer to it as \emph{iterated function system}.

The following fact is well-known (see \cite[Theorem 9.1]{F2003}).

\begin{theorem}\label{UniqueAttractor}
If $\{f_1,\ldots,f_N\}$ is an IFS, then there is a unique {\it attractor}, i.e.\ a non-empty compact set $A_*\subset\mathbb R$ such that
\begin{equation*}
A_*=\bigcup_{n=1}^N f_n(A_*).
\end{equation*}
Moreover, if $A_1=[0,1]$ and $A_{k+1}=\bigcup_{n=1}^{N}f_n(A_{k})$ for every $k\in\mathbb N$, then
\begin{equation}\label{attractor}
A_*=\bigcap_{k\in\mathbb N}A_k.
\end{equation}
\end{theorem}

From the moreover part of \cref{UniqueAttractor} we see that the attractor of the IFS considered in \cref{IFS: question} looks like a set of Cantor type; in fact, \eqref{condition} and the strict monotonicity of $f_1,\ldots,f_N$ imply $A_{k+1}\varsubsetneq A_k$ for every $k\in\mathbb N$.
Let us mention here that not every set of Cantor type can be an attractor of some IFS (see \cite{CR2006}), and moreover, that typical closed sets in $[0,1]$ can not be attractors of any IFS (see \cite{DS2015}). It is also known that the family of all attractors is dense, path connected and an $F_\sigma$ set in the space of all nonempty and compact subsets of $[0,1]$ endowed with the Hausdorff metric (see \cite{Ds2016}).

Note that the strict monotonicity in \cref{IFS: question} is crucial. Indeed, if we omit the word {\lq\lq}strictly{\rq\rq}, then there is no problem to give an example of an IFS whose  attractor is of positive Lebesgue measure.

\begin{example}\label{nonincreasing example}
Define $F,G\colon [0,1]\to [0,1]$ by
\begin{align*}
F(x)=&
\begin{cases}
\frac{1}{2}x, & \mbox{if } x\in [0,\frac{1}{3}] \\
\frac{1}{6}, & \mbox{if } x\in (\frac{1}{3},\frac{2}{3}) \\
\frac{1}{2}x-\frac{1}{6}, & \mbox{if } x\in [\frac{2}{3},1]
\end{cases}
\quad\hbox{and}\quad
G(x)=&
\begin{cases}
\frac{2}{3}+\frac{1}{2}x, & \mbox{if } x\in [0,\frac{1}{3}] \\
\frac{5}{6}, & \mbox{if } x\in (\frac{1}{3},\frac{2}{3}) \\
\frac{1}{2}x+\frac{1}{2}, & \mbox{if }  x\in [\frac{2}{3},1].
\end{cases}
\end{align*}
A short calculation shows that $F([0,\frac{1}{3}]\cup[\frac{2}{3},1])\cup G([0,\frac{1}{3}]\cup[\frac{2}{3},1])=[0,\frac{1}{3}]\cup[\frac{2}{3},1]$. By \cref{UniqueAttractor} the set $[0,\frac{1}{3}]\cup [\frac{2}{3},1]$ is the attractor of the considered IFS, which clearly is not of Cantor type.
\end{example}


\section{The similitudes case}
Now we prove that if the considered IFS consists of similitudes, then the answer to the posed question is positive.

From now on, we denote by $\leb$ the Lebesgue measure on the real line.

\begin{proposition}\label{proposition}
Assume that $f_1,\ldots,f_N\colon[0,1]\to[0,1]$ is an IFS consisting of similitudes satisfying \eqref{condition}. Then the attractor of this IFS is of Lebesgue measure zero.
\end{proposition}

\begin{proof}
For every $n\in\{1,\ldots,N\}$, let the similitude $f_n$ be of the form
\begin{equation*}
f_n(x)=a_n x+b_n
\end{equation*}	
with some $a_n\in(-1,0)\cup(0,1)$ and $b_n\in[0,1]$.
	
Put $q=\leb(A_1\setminus A_2)$ and observe that $1-q=\sum_{n=1}^N|a_n|\in(0,1)$, by \eqref{condition}.
A~simple induction gives $\leb(A_k\setminus A_{k+1})=q(1-q)^{k-1}$ for every $\in\mathbb N$. Hence
\begin{equation*}
\leb(A_*)=1-\sum_{k=1}^{\infty}\leb(A_k\setminus A_{k+1})=1-\frac{q}{1-(1-q)}=0,
\end{equation*}
and the proof is complete.
\end{proof}	


\section{Construction of the example}

We begin with an explanation of the idea how we construct the announced example.
Consider the IFS consisting of the contractions $f_0,g_0\colon [0,1]\to [0,1]$ defined by
\begin{equation*}
f_0(x)=\frac{1}{3}x\quad\hbox{and}\quad g_0(x)=\frac{1}{3}x+\frac{2}{3}.
\end{equation*}
It is well-known that the attractor of this IFS is the Cantor set (see e.g.\ \cite[Chapter 9]{F2003}), which has Lebesgue measure zero. The problem is that the gap $(\frac{1}{3},\frac{2}{3})$ leads to the gaps $(\frac{1}{9},\frac{2}{9})$ and $(\frac{7}{9},\frac{8}{9})$. During the process, the gaps propagate and at the end sum up to a set of Lebesgue measure~$1$. To counteract, we modify the functions $f_0$ and $g_0$. As $(\frac{1}{3},\frac{2}{3})$ and its images generate gaps, we make the gaps smaller by mapping $(\frac{1}{3},\frac{2}{3})$ to smaller sets than $(\frac{1}{9},\frac{2}{9})$ and $(\frac{7}{9},\frac{8}{9})$. We continue to modify the functions $f_0$ and $g_0$ such that the images of the smaller gaps are even smaller. This way, we obtain two sequences of strictly increasing contractions that converge uniformly to strictly increasing contractions that form our IFS\@.


\subsection{Key sequences}
First, we need two sequences $(\varepsilon_k)_{k\in\mathbb N}$ and $(w_k)_{k\in\mathbb N}$ of parameters that will determine how we modify the functions $f_0$ and $g_0$.

We let
\begin{equation*}
w_1=1\quad\hbox{and}\quad \varepsilon_1=\frac{1}{6}.
\end{equation*}
Having defined $\varepsilon_l>0$ and $w_l\in\mathbb R$ for all $1\leq l\leq k$, we let
\begin{equation}\label{width recursive}
w_{k+1}=\frac{w_{k}}{2}-\varepsilon_{k}
\end{equation}
and choose $\varepsilon_{k+1}>0$ such that the following conditions are satisfied
\begin{align}
2^{k}\varepsilon_{k+1} &< \frac{1}{2}\cdot\frac{1}{4^{k}}\label{epsilon geometric},  \\
\frac{\varepsilon_{k+1}}{\varepsilon_{k}} &< \frac{1}{2}\label{epsilon quotient}, \\
\varepsilon_{k+1}&<\frac{w_k}{4}-\frac{\varepsilon_{k}}{2}.\label{width positive}
\end{align}
To see that the sequences $(\varepsilon_k)_{k\in\mathbb N}$ and $(w_k)_{k\in\mathbb N}$ are well-defined, we only have to show that we really can satisfy~\cref{width positive}.
First, we observe that $\frac{w_1}{4}-\frac{\varepsilon_1}{2}=\frac{1}{4}-\frac{1}{12}>0$. Thus we can choose $\varepsilon_2$. Fix $k\in\mathbb N$ and assume that we have already chosen $\varepsilon_{k+1}$ and $w_{k+1}$. Then, using \cref{width recursive} and \cref{width positive}, we have $\frac{w_{k+1}}{4}-\frac{\varepsilon_{k+1}}{2}=\frac{w_{k}}{8}-
\frac{\varepsilon_{k}}{4}-\frac{\varepsilon_{k+1}}{2}>0$, which shows that we can choose $\varepsilon_{k+2}$.

Condition~\cref{epsilon geometric} will be used to show that the attractor of the constructed IFS has positive Lebesgue measure. To guarantee that our functions are contractions, we will need condition~\cref{epsilon quotient}. Finally, conditions~\cref{width positive} and \cref{width recursive} will guarantee that all the intervals where the modifications will take place are non-degenerated but small enough.

\begin{lemma}\label{decay of the width}
For every $k\in\mathbb N$ we have
\begin{equation}\label{wk}
0<w_k\leq\frac{1}{2^{k-1}}
\end{equation}
and
\begin{equation}\label{wk+1}
\frac{w_{k+1}}{w_k}<\frac{1}{2}.
\end{equation}
\end{lemma}

\begin{proof}
Conditions \cref{wk} and \cref{wk+1} are clearly true for $k=1$.	
	
If $k\geq 2$, then using \cref{width recursive} and \cref{width positive}, we get
$w_{k+1}=\frac{w_{k}}{2}-\varepsilon_{k}>2\varepsilon_{k+1}>0$. Thus the first inequality in \cref{wk} is proved. To prove the second one and \cref{wk+1}, it is enough to observe that applying \cref{width recursive} we have
$w_{k+1}=\frac{w_{k}}{2}-\varepsilon_{k}<\frac{w_{k}}{2}$ and proceed by induction on $k$.
\end{proof}


\subsection{Intervals where the modifications take place}
We inductively define a sequence of collections of intervals as follows. Put
\begin{equation*}
\mathcal{I}_1=\{[0,1]\}
\end{equation*}
and observe that the only interval in~$\mathcal{I}_1$ has length~$w_1$.

Fix $k\in\mathbb N$ and assume that the collection $\mathcal{I}_k$ has been defined in such a way that $b-a=w_k$ for each $[a,b]\in\mathcal{I}_{k}$; note that $w_k>0$ by \cref{wk}. Next observe that if $[a,b]\in\mathcal{I}_{k}$, then according to \cref{width recursive} we have
\begin{equation*}
\frac{a+b}{2}-\varepsilon_{k}-a=\frac{w_k}{2}-\varepsilon_{k}=w_{k+1}
\quad\hbox{and}\quad
b-\frac{a+b}{2}-\varepsilon_{k}=\frac{w_k}{2}-\varepsilon_{k}=w_{k+1}.
\end{equation*}
Now we put
\begin{equation*}
\mathcal{I}_{k+1}=\bigcup_{[a,b]\in\mathcal{I}_{k}}\left\{\left[a,\frac{a+b}{2}-\varepsilon_{k}\right],\left[\frac{a+b}{2}+\varepsilon_{k},b\right]\right\}.
\end{equation*}

In this way we have constructed a sequence $(\mathcal{I}_k)_{k\in\mathbb N}$  of collections of intervals.
Let us write down some of the sequence's properties in the next \lcnamecref{properties of Ik}.

\begin{lemma}\label{properties of Ik}
Assume that $k\in\mathbb N$.
\begin{enumerate}
\item The family $\mathcal{I}_k$ consists of $2^{k-1}$ pairwise disjoint closed subintervals of $[0,1]$.
\item If $[a,b]\in\mathcal{I}_k$, then $b-a=w_k$.
\item We have $\bigcup\mathcal{I}_{k+1}\varsubsetneq\bigcup\mathcal{I}_k$.
\item Let $[a,b]\in\mathcal{I}_{k}$. If $[a,b]\subset[0,\frac{1}{3}]$, then
$[a,b]+\frac{2}{3}\in\mathcal{I}_k$, and if $[a,b]\subset[\frac{2}{3},1]$, then $[a,b]-\frac{2}{3}\in\mathcal{I}_k$.
\end{enumerate}
\end{lemma}

\begin{proof}
Assertions (i), (ii) and (iii) are clear from the construction of the sequence $(\mathcal{I}_k)_{k\in\mathbb N}$. Assertion~(iv) can be proved by a simple induction with the fact that
\begin{equation*}
\mathcal{I}_2=\left\{\left[0,\frac{1}{3}\right],\left[\frac{2}{3},1\right]\right\}
\end{equation*} as its first step.
\end{proof}


\subsection{Attractor}
We now define a set, which turns out to be the attractor of our IFS\@.
For every $k\in\mathbb N$ we let
\begin{align*}\label{sets Ak}
A_k=\bigcup \mathcal{I}_k
\end{align*}
and observe that by assertion~(iii) of \cref{properties of Ik} we have \begin{equation}\label{asscending}
A_{k+1}\varsubsetneq A_k.
\end{equation}
Now we define $A_*$ as intersection of all $A_k$ as in \cref{attractor}.

It is clear that
\begin{equation*}
A_*\subset\left[0,\frac{1}{3}\right]\cup\left[\frac{2}{3},1\right].
\end{equation*}
Moreover, assertion~(iv) of \cref{properties of Ik} yields
\begin{equation*}
\left(A_*\cap\left[0,\frac{1}{3}\right]\right)+\frac{2}{3}=A_*\cap\left[\frac{2}{3},1\right].
\end{equation*}

\begin{lemma}\label{nowhere dense}
The set $A_*$ is of Cantor type, i.e. nonempty, compact, perfect and nowhere dense.
\end{lemma}

\begin{proof}
It is easy to see that $0\in A_*$, so $A_*\neq\emptyset$.

From assertion~(i) of \cref{properties of Ik} we conclude that each $A_k$ is closed and bounded. Hence $A_*$ is compact.

For showing that $A_*$ is nowhere dense, suppose the contrary and choose a point
$x_0\in A_*$ and $r>0$ such that $(x_0-r,x_0+r)\subset A_*$. Thus $(x_0-r,x_0+r)\subset A_k$ for every $k\in\mathbb N$, which is impossible by assertions~(i) and (ii) of \cref{properties of Ik} and \cref{wk}.
\end{proof}

\begin{lemma}\label{positive measure}
The set $A_*$ has positive (one-dimensional) Lebesgue measure.
\end{lemma}

\begin{proof}
By \cref{nowhere dense} the set $A_*$ is Lebesgue measurable. We calculate the measure of the complement of $A^*$. In the course of the computation, we need~\cref{asscending}, assertion~(i) of \cref{properties of Ik} and \cref{epsilon geometric}, as well as
\begin{equation*}
[0,1]\setminus A_k\setminus \left(\bigcup_{l=1}^{k-1}\big([0,1]\setminus A_l\big)\right)=A_{k-1}\setminus A_k.
\end{equation*}
Consequently,
\begin{equation*}
\begin{split}
\leb([0,1]\setminus A_*)&=\leb\left([0,1]\setminus \bigcap_{k=2}^\infty A_k\right)=\leb\left(\bigcup_{k=2}^{\infty}\big([0,1]\setminus A_k\big)\right)\\
&=\leb\left(\bigcup_{k=2}^{\infty}\left(\big( [0,1]\setminus A_k\big)\setminus \left(\bigcup_{l=1}^{k-1}\big([0,1]\setminus A_l\big)\right)\right)\right)\\
&\leq\sum_{k=2}^{\infty}\leb\left(\left([0,1]\setminus A_k\right)\setminus \left(\bigcup_{l=1}^{k-1}\big([0,1]\setminus A_l\big)\right)\right)\\
&=\leb\big([0,1]\setminus A_2\big)
+\sum_{k=3}^{\infty}\leb\left(\left([0,1]\setminus A_k\right)\setminus \left(\bigcup_{l=1}^{k-1}\big([0,1]\setminus A_l\big)\right)\right)\\
&=2\varepsilon_1+\sum_{k=3}^{\infty}2^{k-2}\cdot 2\varepsilon_{k-1}
<\frac{1}{3}+\sum_{k=3}^{\infty}\frac{1}{4^{k-2}}=\frac{2}{3}.
\end{split}
\end{equation*}
Finally, we have $\leb(A_*)\geq\frac{1}{3}>0$.
\end{proof}


\subsection{Sequence of functions}
Define the function $f_1\colon [0,1]\to [0,1]$ by
\begin{align*}
f_1(x)&=
\begin{cases}
3(\frac{1}{6}-\varepsilon_2)x, & \mbox{if } x\in [0,\frac{1}{3}] \\
6\varepsilon_2(x-\frac{1}{3})+\frac{1}{6}-\varepsilon_2, & \mbox{if } x\in (\frac{1}{3},\frac{2}{3}) \\
3(\frac{1}{6}-\varepsilon_2)(x-\frac{2}{3})+\frac{1}{6}+\varepsilon_2, & \mbox{if } x\in[\frac{2}{3},1].
\end{cases}
\end{align*}
Note that $f_1$ is a strictly increasing contraction with the minimal Lipschitz constant strictly smaller than $\frac{1}{2}$; here we use that $\varepsilon_2\in(0,\frac{1}{16})$ by~\cref{epsilon geometric} and apply \cref{Lipschitz: from local to global}.
Moreover, simple calculations (some of them with the use of \cref{width recursive}) give
\begin{equation*}
f_1(0)=f_0(0),\quad f_1(1)=f_0(1),\quad
f_1\left(\frac{1}{3}\right)-f_1(0)=f_1(1)-f_1\left(\frac{2}{3}\right)=w_3.
\end{equation*}

Fix $k\in\mathbb N$ and assume that the function $f_{k-1}\colon [0,1]\to [0,1]$ has been defined. Then we define $f_{k}\colon [0,1]\to [0,1]$ as follows.
If $[a,b]\in\mathcal{I}_{k}$, then we define $f_{k}$ on $[a,b]$ by
\begin{equation*}
f_{k}(x)=
\begin{cases}
\frac{w_{k+2}}{w_{k+1}}(x-a)+f_{k-1}(a), & \mbox{if } x\in [a,\frac{a+b}{2}-\varepsilon_{k}] \\
\frac{\varepsilon_{k+1}}{\varepsilon_{k}}(x-\frac{a+b}{2}+\varepsilon_{k})+f_{k-1}(a)+w_{k+2}, & \mbox{if }|x-\frac{a+b}{2}|<\varepsilon_{k} \\
\frac{w_{k+2}}{w_{k+1}}(x-\frac{a+b}{2}-\varepsilon_{k})+f_{k-1}(a)+w_{k+2}+2\varepsilon_{k+1}, &  \mbox{if }x\in [\frac{a+b}{2}+\varepsilon_{k},b];
\end{cases}
\end{equation*}
see \cref{Graphs}. Note that the above formula is consistent with the definition of~$f_1$ by \cref{width recursive} and simple calculations. In this way we have defined $f_k$ on $\bigcup\mathcal{I}_{k}$. Now we put
\begin{equation*}
f_{k}(x)=f_{k-1}(x)\quad\hbox{for every }x\in[0,1]\setminus\bigcup\mathcal{I}_{k}.
\end{equation*}

\tikzmath{\epsi=1/100;}
\begin{figure}
\begin{tikzpicture}[
declare function={
func(\x) =(\x<0.333333)*3*(1.0/6-\epsi)*\x + and(0.333333<\x,\x<0.6666666)*(6*\epsi*(\x-1/3)+1/6-\epsi) +(\x>0.6666666) *(3*(1.0/6-\epsi)*(x-2.0/3)+1.0/6+\epsi);}]
\begin{axis}[
xtick={0,1},
xticklabels={$a$\vphantom{$b$}, $b$},
extra x ticks={0.33333333,0.666666666},
extra x tick labels={$\!\!\!\!\!\!\!\frac{a+b}{2}\!-\!\varepsilon_k$,$\,\,\,\,\,\,\,\frac{a+b}{2}\!+\!\varepsilon_k$},
extra x tick style={
    xticklabel style={yshift=-1cm, anchor=south}
},
ytick={0, 0.15666666, 0.17666666, 0.3333333333},
yticklabels={$f_{k}(a)$,\raisebox{-.5cm}{$f_{k}(\frac{a+b}{2}\!-\!\varepsilon_k)$},\raisebox{.5cm}{$f_{k}(\frac{a+b}{2}\!+\!\varepsilon_k)$},$f_{k}(b)$},width=.41\textwidth,
height=.41\textwidth,
domain=0:1,]
\addplot[blue] {func(x)};
\end{axis}
\end{tikzpicture}
\quad
\begin{tikzpicture}[
declare function={
func(\x) =\x/3;}]
\begin{axis}[
xtick={0,1},
xticklabels={$a$\vphantom{$b$}, $b$},
extra x ticks={0.33333333,0.666666666},
extra x tick labels={$\!\!\!\!\!\!\!\frac{a+b}{2}\!-\!\varepsilon_k$,$\,\,\,\,\,\,\,\frac{a+b}{2}\!+\!\varepsilon_k$},
extra x tick style={
    xticklabel style={yshift=-1cm, anchor=south}
},
ytick={0, 0.1111111, 0.22222222222222, 0.3333333333},
yticklabels={$f_{k-1}(a)$,$f_{k-1}(\frac{a+b}{2}\!-\!\varepsilon_k)$,$f_{k-1}(\frac{a+b}{2}\!+\!\varepsilon_k)$,$f_{k-1}(b)$},
width=.41\textwidth,
height=.41\textwidth,
domain=0:1,]
\addplot[blue] {func(x)};
\end{axis}
\end{tikzpicture}
\caption{Graphs of $f_k$ and $f_{k-1}$ on $[a,b]\in\mathcal{I}_{k}$}\label{Graphs}
\end{figure}
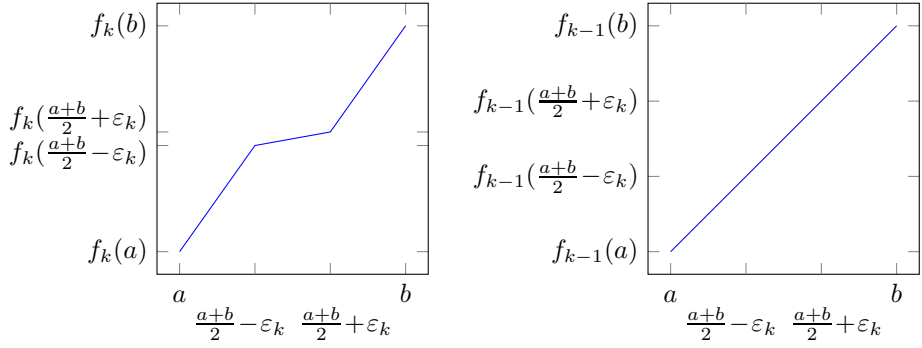
The next lemma collects essential properties of the just defined sequence $(f_k)_{k\in\mathbb N}$.

\begin{lemma}\label{properties of fk}
Assume that $k\in\mathbb N$.
\begin{enumerate}
\item If $[a,b]\in\mathcal{I}_{k+1}$, then $f_{k}(b)-f_{k}(a)=w_{k+2}$.
\item If $[a,b]\in\mathcal{I}_{k}$, then $f_{k}(a)=f_{k-1}(a)$ and $f_{k}(b)=f_{k-1}(b)$.
\item The function $f_{k-1}$ is affine on each $[a,b]\in\mathcal{I}_{k}$.
\item The function $f_k$ is strictly increasing.
\item The function $f_k$ is a contraction with Lipschitz constant strictly smaller than $\frac{1}{2}$.
\item We have $f_k([0,1])=[0,\frac{1}{3}]$.
\item If $[a,b]\in\mathcal{I}_{k}$, then  $[f_{k-1}(a),f_{k-1}(b)]\in\mathcal{I}_{k+1}$.
\item If $[a,b]\in\mathcal{I}_{k+1}$, then there exists an interval $[c,d]\in\mathcal{I}_{k}$ such that either $[a,b]=[f_{k-1}(c),f_{k-1}(d)]$ or $[a,b]=[f_{k-1}(c),f_{k-1}(d)]+\frac{2}{3}$.
\end{enumerate}
\end{lemma}

\begin{proof}
(i) Fix an interval $[c,d]\in\mathcal{I}_{k+1}$. Then there exists an interval $[a,b]\in\mathcal{I}_{k}$ such that either $[c,d]=[a,\frac{a+b}{2}-\varepsilon_{k}]$ or $[c,d]=[\frac{a+b}{2}+\varepsilon_{k},b]$.

If $[c,d]=[a,\frac{a+b}{2}-\varepsilon_{k}]$, then using assertion~(ii) of \cref{properties of Ik} and \cref{width recursive} we obtain
\begin{equation}\label{value at}
\begin{split}
f_{k}\left(\frac{a+b}{2}-\varepsilon_{k}\right)&=
\frac{w_{k+2}}{w_{k+1}}\left(\frac{b-a}{2}-\varepsilon_{k}\right)+f_{k-1}(a)\\
&=\frac{w_{k+2}}{w_{k+1}}\left(\frac{w_{k}}{2}-\varepsilon_{k}\right)+f_{k-1}(a)
=w_{k+2}+f_{k-1}(a).
\end{split}
\end{equation}
Hence $f_{k}(d)-f_{k}(c)=w_{k+2}$.

If $[c,d]=[\frac{a+b}{2}+\varepsilon_{k+1},b]$, then the same arguments as above give
\begin{equation*}
f_{k}(d)-f_{k}(c)=f_{k}(b)-f_{k}\left(\frac{a+b}{2}+\varepsilon_{k}\right)=\frac{w_{k+2}}{w_{k+1}}\left(\frac{b-a}{2}-\varepsilon_{k}\right)=w_{k+2}.
\end{equation*}

(ii) Fix $[a,b]\in\mathcal{I}_{k}$. We see at once that $f_{k}(a)=f_{k-1}(a)$. Applying assertion~(ii) of \cref{properties of Ik}, the just proven assertion~(i) and \cref{width recursive}, we obtain
\begin{equation*}
\begin{split}
f_{k}(b)&=\frac{w_{k+2}}{w_{k+1}}\left(\frac{b-a}{2}-\varepsilon_{k}\right)+f_{k-1}(a)+w_{k+2}+2\varepsilon_{k+1}\\
&=\frac{w_{k+2}}{w_{k+1}}\left(\frac{w_{k}}{2}-\varepsilon_{k}\right)+f_{k-1}(b)-w_{k+1}+w_{k+2}+2\varepsilon_{k+1}\\
&=2w_{k+2}+f_{k-1}(b)-w_{k+1}+2\varepsilon_{k+1}=f_{k-1}(b).
\end{split}
\end{equation*}

(iii) Clearly, $f_0$ is affine on $[0,1]$. Then a simple induction completes the proof.

(iv) It is enough to observe that on each interval $[a,b]\in\mathcal{I}_{k}$ the function $f_k$ is strictly increasing by \cref{wk}, and then proceed by induction on $k$ using assertion (ii).

(v) We first observe that $f_k$ is continuous on any interval $[a,b]\in\mathcal{I}_{k}$; continuity at $\frac{a+b}{2}+\varepsilon_{k}$ is clear and continuity at $\frac{a+b}{2}-\varepsilon_{k}$ follows from \cref{value at}.
Then combining \cref{wk+1} with \cref{epsilon quotient} and applying \cref{Lipschitz: from local to global} jointly with assertion (ii) we conclude that $f_k$ restricted to any interval $[a,b]\in\mathcal{I}_{k}$ is $L$-Lipschitz with $L<\frac{1}{2}$. Now the assertion can be proved by induction on $k$.

(vi) From the construction we can easy conclude, proceeding by induction with the use of assertion (ii), that for every $k\in\mathbb N$ we have $f_k(0)=0$ and $f_k(1)=\frac{1}{3}$.

(vii)
The assertion is clear for $k=1$.

Fix $k\in\mathbb N$ and assume inductively that (vii) holds. Fix also an interval $[c,d]\in\mathcal{I}_{k+1}$. Then there exists an interval $[a,b]\in\mathcal{I}_k$ such that either $[c,d]=[a,\frac{a+b}{2}-\varepsilon_{k}]$ or $[c,d]=[\frac{a+b}{2}+\varepsilon_{k},b]$.

First, we consider the case where $[c,d]=[a,\frac{a+b}{2}-\varepsilon_{k}]$.
Assertion~(ii) gives $f_{k}(c)=f_{k}(a)=f_{k-1}(a)$. Then assertion~(i) and the induction hypothesis imply $[f_{k}(c),f_{k}(c)+w_{k+1}]=[f_{k-1}(a),f_{k-1}(a)+w_{k+1}]=[f_{k-1}(a),f_{k-1}(b)]\in\mathcal{I}_{k+1}$. Hence, by the definition of $\mathcal{I}_{k+2}$, we see that
$[f_{k}(c),f_{k}(c)+\frac{w_{k+1}}{2}-\varepsilon_{k+1}]\in\mathcal{I}_{k+2}$.
Finally, according to assertion~(i) and \cref{width recursive}  we conclude that
\begin{equation*}
[f_{k}(c),f_{k}(d)]=[f_{k}(c),f_{k}(c)+w_{k+2}]=\left[f_{k}(c),f_{k}(c)+\frac{w_{k+1}}{2}-\varepsilon_{k+1}\right]\in\mathcal{I}_{k+2}.
\end{equation*}

Now, we consider the case where $[c,d]=[\frac{a+b}{2}+\varepsilon_{k},b]$.
Assertion~(ii) gives $f_{k}(d)=f_{k}(b)=f_{k-1}(b)$. Then assertion~(i) and the induction hypothesis imply $[f_{k}(d)-w_{k+1},f_{k}(d)]=[f_{k-1}(b)-w_{k+1},f_{k-1}(b)]=[f_{k-1}(a),f_{k-1}(b)]\in\mathcal{I}_{k+1}$. Hence, by the definition of $\mathcal{I}_{k+2}$, we see that $[f_{k}(d)-\frac{w_{k+1}}{2}+\varepsilon_{k+1},f_{k}(d)]\in\mathcal{I}_{k+2}$.
Finally, according to assertion~(i) and \cref{width recursive}  we conclude that
\begin{equation*}
[f_{k}(c),f_{k}(d)]=[f_{k}(d)-w_{k+2},f_{k}(d)]=\left[f_{k}(d)-\frac{w_{k+1}}{2}+\varepsilon_{k+1},f_{k}(d)\right]\in\mathcal{I}_{k+2}.
\end{equation*}

(viii) From assertion~(i) of \cref{properties of Ik} we see that families $\mathcal{I}_{k}$ and $\mathcal{I}_{k+1}$ consist of $2^{k-1}$ and $2^k$ pairwise disjoint closed intervals, respectively. The just proved assertion~(vii) and assertion~(iv) of \cref{properties of Ik} imply that with each interval $[a,b]\in\mathcal{I}_k$ there are associated exactly two intervals of the form
$[f_{k-1}(a),f_{k-1}(b)]$ and $[f_{k-1}(a),f_{k-1}(b)]+\frac{2}{3}$, and both belong to $\mathcal{I}_{k+1}$. It remains to note that if $[a,b]$ and $[c,d]$ are different intervals from $\mathcal{I}_k$, then the intervals associated with them
are pairwise disjoint, by assertions (iv) and (vi).

The proof is complete.
\end{proof}


\subsection{The limit function}
We show that the sequence $(f_k)_{k\in\mathbb N}$ converges pointwise to a strictly increasing contraction. We begin with showing that it is convergent.

\begin{lemma}\label{Cauchy}
The sequence $(f_k)_{k\in\mathbb N}$ is a Cauchy sequence with respect to the supremum norm.
\end{lemma}

\begin{proof}
Fix $k\in\mathbb N$, interval $[a,b]\in\mathcal{I}_{k}$ and note that we only need to show that $\sup\{|f_{k}(x)-f_{k-1}(x)|:x\in[a,b]\}\leq\frac{1}{2^k}$. According to assertion~(ii) of \cref{properties of fk} the supremum is attained at the point $\frac{a+b}{2}-\varepsilon_k$ (see \cref{Graphs}). Therefore, it suffices to prove that
$M=|f_k(\frac{a+b}{2}-\varepsilon_k)-f_{k-1}(\frac{a+b}{2}-\varepsilon_k)|\leq\frac{1}{2^k}$.

Making use of \cref{value at}, assertions (iii) and~(i) of \cref{properties of fk}, assertion~(ii) of \cref{properties of Ik}, \cref{wk+1}, \cref{width recursive}, and \cref{decay of the width}  we get
\begin{equation*}
\begin{split}
M&=\left|w_{k+2}+f_{k-1}(a)-f_{k-1}\left(\frac{a+b}{2}-\varepsilon_k\right)\right|\\
&=\left|w_{k+2}+f_{k-1}(a)-\frac{f_{k-1}(b)-f_{k-1}(a)}{b-a}\left(\frac{a+b}{2}-\varepsilon_k-a\right)-f_{k-1}(a)\right|\\
&=\left|w_{k+2}-\frac{w_{k+1}}{w_k}\left(\frac{w_k}{2}-\varepsilon_k\right)\right|\leq w_{k+2}+\frac{1}{2}{w_{k+1}}\leq\frac{1}{2^{k}}.
\end{split}
\end{equation*}
The proof is complete.
\end{proof}

Define the function $f\colon[0,1]\to[0,1]$ by
\begin{equation*}
f(x)=\lim_{k\to\infty}f_k(x);
\end{equation*}
\cref{Cauchy} shows that $f$ is well-defined and continuous.

\begin{lemma}\label{strictly increasing}
The function $f$ is strictly increasing.
\end{lemma}

\begin{proof}
The function $f$ is increasing by assertion~(iv) of \cref{properties of fk}.

Suppose the assertion of the lemma is false. Then there exists an interval $[x,y]\subset[0,1]$ on which $f$ is constant.

By \cref{nowhere dense}, the set $A_*$ is closed and nowhere dense. Hence, we  find a point $z\in (x,y)$ and $r>0$ such that $[z-r,z+r]\subset(x,y)$ and $[z-r,z+r]\cap A_*=\emptyset$. Since  the sequence $(A_k)_{k}$ is descending, we see that  there exists $k\in \field{N}$ such that $[z-r,z+r]\cap A_k=\emptyset$. This implies that $f$ equals $f_k$ on~$[z-r,z+r]$. Therefore, $f_k$ is constant on~$[z-r,z+r]$, which contradicts that $f_k$ is strictly increasing as pointed out in assertion~(iv) of \cref{properties of fk}.
\end{proof}

\begin{lemma}\label{contraction}
The function $f$ is a contraction.
\end{lemma}

\begin{proof}
Note that, by assertion~(v) of \cref{properties of fk}, we have 
\begin{equation*}
  \abs{f(y)-f(x)}=\lim_{k\to \infty} \abs{f_k(y)-f_k(y)}\leq\frac{1}{2}\abs{x-y}
\end{equation*} 
for all $x,y\in[0,1]$.
\end{proof}

We finish this subsection with proving a property of $f$, which will be used later.

\begin{lemma}\label{vaule at endpoints}
If $x$ is an endpoint of an interval belonging to $\mathcal{I}_k$ for some $k\in\mathbb N$, then $f(x)=f_{k-1}(x)$.
\end{lemma}

\begin{proof}
Fix $k\in\mathbb N$ and an interval $[a,b]\in\mathcal{I}_k$. A trivial verification shows that for every $l\geq k$ the point $a$ is always a left endpoint of an interval from $\mathcal{I}_{l}$ and the point $b$ is always a right endpoint of an interval from $\mathcal{I}_{l}$. This jointly with assertion~(ii) of \cref{properties of fk} implies $f_{l}(a)=f_{k-1}(a)$ and $f_{l}(b)=f_{k-1}(b)$ for every $l\geq k$, and hence $f(a)=f_{k-1}(a)$ and $f(b)=f_{k-1}(b)$.
\end{proof}


\subsection{Definition of the IFS}
We define the announced IFS by taking $f$ and $g=f+\frac{2}{3}$.

By assertion~(vi) of \cref{properties of fk} we have
\begin{equation*}
f([0,1])=\left[0,\frac{1}{3}\right]\quad\hbox{and}\quad  g([0,1])=\left[\frac{2}{3},1\right].
\end{equation*}
Moreover, from \cref{strictly increasing} and \cref{contraction} we see that our IFS consists of strictly increasing contractions.

If we show that its attractor is the set $A_*$, then the example will be complete.

\begin{lemma}
The set $A_*$ is the attractor of the IFS consisting of $f$ and $g$.
\end{lemma}

\begin{proof}
We first prove that
\begin{equation}\label{Ak+1}
A_{k+1}=f(A_k)\cup g(A_k)
\end{equation}
for every $k\in\mathbb N$.

Fix $k\in\mathbb N$. From assertion~(viii) of \cref{properties of fk} we conclude that $A_{k+1}\subset f(A_k)\cup g(A_k)$.
If we prove that $f(A_k)\cup g(A_k)\subset A_{k+1}$, the assertion follows.

Fix $x\in A_k$ and choose an interval $[a,b]\in\mathcal{I}_k$ such that $x\in[a,b]$. From \cref{vaule at endpoints} we get $f(a)=f_{k-1}(a)$ and $f(b)=f_{k-1}(b)$. Then using \cref{strictly increasing} and assertion~(vii) of \cref{properties of fk} we obtain
\begin{equation*}
f(x)\in[f(a),f(b)]=[f_{k-1}(a),f_{k-1}(b)]\subset\bigcup\mathcal{I}_{k+1}=A_{k+1}.
\end{equation*}
Making also use of assertion~(vi) of \cref{properties of fk} and assertion~(iv) of \cref{properties of Ik} we get
\begin{equation*}
g(x)=f(x)+\frac{2}{3}\in[f(a),f(b)]+\frac{2}{3}
\subset\bigcup\mathcal{I}_{k+1}=A_{k+1},
\end{equation*}
which proves \cref{Ak+1}.

In \cref{nowhere dense} we have recorded already that $A_*$ is nonempty and compact. According to \cref{UniqueAttractor} it remains to prove that $A_*=f(A_*)\cup g(A_*)$.

As $f$ and $g$ are strictly increasing as verified in the proof of~\cref{strictly increasing}, we have, using \cref{asscending},
\begin{equation*}
f(A_*)\cup g(A_*)=\bigcap_{k\in\field{N}}f\big(A_k\big)\cup \bigcap_{k\in \field{N}}g\big(A_k)=\bigcap_{k\in\field{N}}\big(f(A_k)\cup g(A_k)\big)
=\bigcap_{k\in\field{N}}A_{k+1}=A_*.
\end{equation*}
The proof is complete.
\end{proof}


\section*{Acknowledgment}
The research of was supported by the Silesian University Mathematics Department (Iterative Functional Equations and Real Analysis program).
Furthermore, the research leading to these results has received funding from the European Research~Council under the European Union's Seventh Framework Programme (\ignoreSpellCheck{FP}/2007-2013) / \ignoreSpellCheck{ERC} Grant Agreement n.291497 while the second author was a research fellow at the University of Warwick.

\bibliographystyle{plain}
\bibliography{PositiveMeasure}

\end{document}